\documentclass{amsart}
\usepackage{graphicx}
\usepackage{amsfonts}
\usepackage{float}
\newtheorem{tm}{Theorem}

\newtheorem{defi}{Definition}

\newtheorem{rem}{Remark}
\newtheorem{rems}{Remarks}
\newtheorem{lm}{Lemma}
\newtheorem{ex}{Example}

\newtheorem{prop}{Proposition}
\newtheorem{nota}{Notation}

\newtheorem{prob}{Problem}

\begin{document}

\title{On a problem inspired by Descartes' rule of signs}
\author{Vladimir Petrov Kostov}

%\address{Universit\'e de Carthage, EPT-LIM, Tunisie}
%\email{yousra.gati@gmail.com}
\address{Universit\'e C\^ote d’Azur, CNRS, LJAD, France}
\email{vladimir.kostov@unice.fr}
%\address{Universit\'e de Carthage, EPT-LIM, Tunisie}
%\email{mohamedchaouki.tarchi@gmail.com}

\begin{abstract}
  We study real univariate polynomials with non-zero coefficients and with all
  roots real, out of which exactly two positive. The sequence of
  coefficients of such a polynomial begins with $m$ positive coefficients
  followed by $n$ negative followed by $q$ positive coefficients. We consider
  the sequence of moduli of their roots on the positive real half-axis;
  all moduli are supposed distinct. 
  We mark in this sequence the
  positions of the moduli of the two positive roots.
  For $m=n=2$, $n=q=2$ and $m=q=2$, we give
  the exhaustive answer to the question which the positions of the two
  moduli of positive roots can be.


\end{abstract}
\maketitle 

\section{Introduction}

The present paper treats a problem inspired by Descartes' rule of signs.
The latter says that given a real univariate polynomial $Q:=\sum _{j=0}^da_jx^j$,
$a_d\neq 0$, the number $r_+$ of its positive roots
(counted with multiplicity) is
majorized by the number $\tilde{c}$ of the sign changes in the sequence $S$ 
of its coefficients and the difference $\tilde{c}-r_+$ is even.
(About Descartes' rule of signs see
\cite{Ca}, \cite{Cu}, \cite{DG},
\cite{Des}, \cite{Fo}, \cite{Ga}, \cite{J}, \cite{La} or \cite{Mes}.)

We are interested in the case when all coefficients $a_j$ are non-zero and
the polynomial $Q$ is {\em hyperbolic}, i.~e. all its roots are real. In this
case one has $\tilde{c}=r_+$ and $\tilde{p}=r_-$, where $\tilde{p}$ is the
number of sign preservations in the sequence $S$ and $r_-$
is the number of negative roots of $Q$. Clearly
$\tilde{c}+\tilde{p}=r_++r_-=d$.

\begin{defi}
  {\rm A {\em sign pattern} is a vector whose components equal $+$ or $-$.
    The polynomial $Q$ is said to define the sign pattern 
    $\sigma (Q):=({\rm sgn}(a_d),{\rm sgn}(a_{d-1}),\ldots ,{\rm sgn}(a_0))$.
  We focus mainly on monic polynomials in which case sign patterns begin
  with a~$+$.}
\end{defi}

Consider the moduli of the roots of a hyperbolic polynomial
as a sequence of $d$ points on
the positive half-axis. We study the {\em generic} case when all these
moduli are distinct. One can mark in this sequence the
$\tilde{p}$ positions of the moduli of negative and the $\tilde{c}$ positions
of moduli of positive roots. This defines the {\em order of moduli}
whose definition and notation should be clear from the following example:

\begin{ex}
  {\rm If for the positive roots $\alpha _1<\alpha _2<\alpha _3$ and the
    moduli of the negative roots $|-\gamma _1|<\cdots <|-\gamma _4|$ of a
    degree $7$ hyperbolic polynomial one has}

  $$\alpha _1<|-\gamma _1|<|-\gamma _2|<\alpha _2<|-\gamma _3|<\alpha _3<
  |-\gamma _4|~,$$
  {\rm then these moduli define the order $PNNPNPN$ (the letters $P$ and $N$
    refer to the relative positions of the moduli of positive and
    negative roots).}
  \end{ex}

\begin{defi}
  {\rm A couple (sign pattern, order of moduli) is {\em compatible with
      Descartes' rule of signs} if the number of letters $P$ (resp. $N$) in
    the order equals $\tilde{c}$ (resp. $\tilde{p}$). In what follows we
    consider only couples compatible with
    Descartes' rule of signs.}
  \end{defi}

We study the following problem:

\begin{prob}\label{prob1}
  Consider the class of hyperbolic polynomials defining one and the same
  sign pattern $\sigma$. What are the possible orders of moduli for the
  polynomials of this class?
  \end{prob}

The problem has been completely resolved for $\tilde{c}=0$ and~1,
see~\cite{KoPuMaDe}. The results of this paper concern the case
$\tilde{c}=2$. We use the following
notation:

\begin{nota}
  {\rm (1) For $\tilde{c}=1$, 
    (resp. $\tilde{c}=2$), we denote by $\Sigma _{m,n}$, $m+n=d+1$, 
    (resp. $\Sigma _{m,n,q}$, $m+n+q=d+1$) the sign pattern consisting of
    $m$ signs $+$
    followed by $n$ signs $-$ (resp. of $m$ signs $+$
    followed by $n$ signs $-$ followed by $q$ signs~$+$).
    \vspace{1mm}
    
    (2) For a polynomial $Q$ with $\sigma (Q)=\Sigma _{m,n}$, we denote by
    $\alpha$ its positive and by $\gamma _1<\cdots <\gamma _{d-1}$ the moduli of
    its negative roots. If $\gamma _u<\alpha <\gamma _{u+1}$
    (we set $\gamma _0:=0$ and 
    $\gamma _d:=+\infty$), then we denote
    the given couple (sign pattern, order of moduli) by
    $(\Sigma _{m,n}, (u,v))$, $v=d-1-u$.
    \vspace{1mm}
    
(3) For a polynomial $Q$ with $\sigma (Q)=\Sigma _{m,n,q}$, we denote by
    $\beta <\alpha$ its positive and by
    $\gamma _1<\cdots <\gamma _{d-2}$ the moduli of
    its negative roots. If $\gamma _u<\beta <\gamma _{u+1}$ and
    $\gamma _{u+v}<\alpha <\gamma _{u+v+1}$, $v\geq 0$ (we set $\gamma _0:=0$ and 
    $\gamma _{d-1}:=+\infty$), then we denote
    the given couple (sign pattern, order of moduli) by
    $(\Sigma _{m,n,q}, (u,v,w))$, $w=d-2-u-v$.}
\end{nota}

\begin{defi}
  {\rm If for a given sign pattern there exists a hyperbolic polynomial
    defining this sign pattern, then we say that the polynomial
    {\em realizes} the sign pattern. If, in addition, the
    roots of the polynomial define a given order of moduli, then we say that
    the polynomial realizes the given couple (sign pattern,
    order of moduli) or that the order of moduli is realizable with the given
    sign pattern.}
  \end{defi}

We can now reformulate Problem~\ref{prob1}:

\begin{prob}\label{prob2}
  For a given degree $d$, which couples (sign pattern, order of moduli)
  are realizable?
  \end{prob}

\begin{defi}\label{defiZ2Z2}
  {\rm (1) For a given degree $d$, we define the following
    two commuting involutions acting 
    on the set of couples:}

  $$i_m~:~Q(x)\mapsto (-1)^dQ(-x)~~~\, {\rm and}~~~\, i_r~:~Q(x)\mapsto
  x^dQ(1/x)/Q(0)~.$$
  {\rm The factors $(-1)^d$ and $1/Q(0)$ are introduced to preserve the set
    of monic polynomials. The involution $i_r$ reads orders, 
    sign patterns and polynomials (modulo the factor $1/Q(0)$) from the right
    while preserving the quantities $\tilde{c}$ and~$\tilde{p}$. In
    particular, $i_r((\Sigma _{m,n},(u,v)))=(\Sigma _{n,m},(v,u))$ and
    $i_r((\Sigma _{m,n,q},(u,v,w)))=(\Sigma _{q,n,m},(w,v,u))$. 
    The involution $i_m$ changes the signs of the odd (resp. even) monomials for
    $d$ even (resp. for $d$ odd). It exchanges the letters
    $P$ and $N$ in the order of moduli and the quantities $\tilde{c}$ and
    $\tilde{p}$, therefore when answering Problem~\ref{prob2} it suffices
    to study the cases with $\tilde{c}\leq \tilde{p}$. 
    \vspace{1mm}

    (2) The {\em orbits} of couples under the
    $\mathbb{Z}_2\times \mathbb{Z}_2$-action
    are of length $4$ or $2$. Orbits of length $2$ can occur only for
    sign patterns $\sigma$ such that $\sigma =i_r(\sigma )$ or
    $\sigma =i_ri_m(\sigma )$; $\sigma =i_m(\sigma )$ is impossible.
    One can consider orbits also only of sign
    patterns or of orders of moduli. All couples of a given orbit are
  simultaneously (non)-realizable.}
\end{defi}

\begin{rems}\label{remscanon}
{\rm (1) Problem~\ref{prob2} is completely resolved for $d\leq 6$,
see~\cite{GaKoTa} and~\cite{Kodeg6}. For $\tilde{c}=2$,
it is settled for $n=1$, see
\cite[Theorem~5]{KoPuMaDe}, and for $q=1$ (hence for $m=1$ as well),
see~\cite{KoSerd23}.
\vspace{1mm}

(2) Each sign pattern is realizable with its {\em canonical} order of
moduli, see \cite[Definition~2 and Proposition~1]{KoSe}.
For the sign pattern $\Sigma _{m,n,q}$, the couple with the corresponding
canonical order is $(\Sigma_{m,n,q},(q-1,n-1,m-1))$. There exist sign patterns
(called also {\em canonical}) which are realizable only with their
corresponding canonical orders, see \cite[Theorem~7]{KoRM}. Among the sign
patterns of the nform $\Sigma_{m,n,q}$, canonical are only $\Sigma_{1,n,1}$
and $\Sigma_{m,1,q}$. 
The relative part of
canonical sign patterns within the set of all
sign patterns of a given degree $d$, tends to $0$
as $d$ tends to $\infty$, see \cite[Definition~9 and Proposition~10]{KoRM}.

(3) There exist also orders of moduli
(called {\em rigid}) realizable with a single sign pattern, see
\cite[Definition~6, Notation~7 and Theorem~8]{KorigMO}. } 
\end{rems}

The first result of the present paper 
about couples $(\Sigma _{m,n,q},(u,v,w))$ reads:

\begin{tm}\label{tmSm22}
  (1) Suppose that $d\geq 7$ and $n=q=2$. Then realizable can be only
  couples with $w\geq m-3$. For each $d\geq 7$ fixed, there are $15$
  triples $(u,v,w)$ satisfying the latter condition.

  (2) For $d\geq 6$, the $10$ triples $(u,v,w)$ with $u+v\leq 3$ and the
  triple $(0,4,m-3)$, are realizable
  with the sign pattern
  $\Sigma _{m,2,2}$, $d=m+3$.

  (3) For $d\geq 6$, the triples $(4,0,m-3)$, $(3,1,m-3)$, $(2,2,m-3)$ and
  $(1,3,m-3)$
  are not realizable  with the
  sign pattern $\Sigma _{m,2,2}$.
\end{tm}

The theorem is proved in Section~\ref{secprtmSm22}. 
It can be automatically reformulated for the case $m=n=2$ using the involution
$i_r$, see Definition~\ref{defiZ2Z2}, and for certain couples with
$\tilde{p}=2$ with the help of the 
involution~$i_m$.

Our second result is formulated as follows:

\begin{tm}\label{tmS2n2}
  (1) For $n\geq 4$, all couples $(\Sigma _{2,n,2},(u,v,w))$ with $u\leq 2$
  and $w\leq 2$ are realizable.
  \vspace{1mm}
  
  (2) For $n\geq 4$, all couples $(\Sigma _{2,n,2},(u,v,w))$ with $u\geq 3$ or
  $w\geq 3$ are not realizable.
\end{tm}

The theorem is proved in Section~\ref{secprtmS2n2}.

\section{Proof of Theorem~\protect\ref{tmSm22}\protect\label{secprtmSm22}}

\begin{proof}
  Part (1). We set $\alpha :=1$ which can be obtained by a linear change of
  the variable~$x$. The set of monic hyperbolic polynomials defining the sign
  pattern $\Sigma _{d-3,2,2}$ is open and connected (see \cite[Theorem~2]{KoAnn})
  and there exists a hyperbolic 
  polynomial with this sign pattern and with $w=m-1$, see part(2) of
  Remarks~\ref{remscanon}.
  Hence if there exists such a
  polynomial with $w\leq m-4$, then for this sign pattern,
  there exists a hyperbolic polynomial
  which is of the form

  $$\begin{array}{ccl}Q&:=&(x^2-1)(x-\beta )(x^{d-3}+e_1x^{d-4}+\cdots +e_{d-3})
    \\ \\
    &=&(x^2-1)(x^{d-2}+(e_1-\beta )x^{d-3}+(e_2-\beta e_1)x^{d-4}+
    (e_3-\beta e_2)x^{d-5}+\cdots \\ \\ &&+(e_{d-5}-\beta e_{d-6})x^3+
    (e_{d-4}-\beta e_{d-5})x^2+(e_{d-3}-\beta e_{d-4})x-\beta e_{d-3})~.
  \end{array}$$
  Hence one of the moduli of negative roots of the polynomial $Q$ equals $1$
  and we suppose that exactly four of its moduli of negative roots are smaller
  than $1$. The second factor in the right-hand side
  defines the sign pattern $\Sigma _{d-3,2}$.
  Indeed, this factor has exactly one positive root, so its sign pattern is
  of the form $\Sigma _{d-1-\nu ,\nu}$, $0<\nu <d-1$. Then the coefficient of
  $x^{\nu +1}$ of $Q$
  is negative, so $\nu +1\leq 3$, i.~e. $\nu \leq 2$. For $\nu =1$,
  the coefficient of $x$ is negative, so $\nu =2$. Thus
  by setting $Q:=\sum _{j=0}^dc_jx^j$, one obtains the conditions

  \begin{equation}\label{eq4cond}
    \begin{array}{ll}e_{d-4}>\beta e_{d-5}~,&e_{d-3}<\beta e_{d-4}\\ \\
      c_4:=e_{d-4}-\beta e_{d-5}-e_{d-6}+\beta e_{d-7}>0~,&
      c_1:=-e_{d-3}+\beta e_{d-4}>0~.
  \end{array}\end{equation}
  The fourth of conditions (\ref{eq4cond}) and the inequality

  $$c_3:=(e_{d-3}-\beta e_{d-4})-(e_{d-5}-\beta e_{d-6})<0$$
  are
  corollaries of the second of conditions (\ref{eq4cond}); for $c_3<0$,
  one has to use also $e_{d-5}-\beta e_{d-6}>0$
  which results from the sign pattern
  $\Sigma _{d-3,2}$. We denote by

  $$r_1>\cdots >r_{d-3}$$
  the moduli of negative roots different from $-1$. Hence we suppose that
  $r_{d-7}>1>r_{d-6}$. We set

  $$\begin{array}{ccllc}
    E_k&:=&\sum _{1\leq i_1<\cdots <i_k\leq d-7}r_{i_1}\cdots r_{i_k}~,&1\leq k\leq d-7&
    {\rm and}\\ \\ 
    e_s'&:=&\sum _{d-6\leq p_1<\cdots <p_s\leq d-3}r_{p_1}\cdots r_{p_s}~,&1\leq s\leq 4~.
  \end{array}$$
  We set $E_k:=0$ for $k\leq 0$ and $k\geq d-6$, and $e_s':=0$ for $s\leq 0$
  and $s\geq 5$. Using this notation one can write

  $$\begin{array}{cclc}
    e_{d-6}&=&E_{d-7}e_1'+E_{d-8}e_2'+E_{d-9}e_3'+E_{d-10}e_4'&{\rm and}\\ \\
    e_{d-4}&=&E_{d-7}e_3'+E_{d-8}e_4'~.&\end{array}$$
  It is clear that $e_1'>e_3'$, because these are elementary symmetric
  polynomials having the same number (namely 4) of terms and their arguments
  are in the interval $(0,1)$. It is also evident that $e_2'>e_4'$ (their
  numbers of terms are $6$ and $1$), $e_3'>0$ and $e_4'>0$. Therefore
  $e_{d-6}>e_{d-4}$. If
  $e_{d-5}\geq e_{d-7}$, then (see (\ref{eq4cond})) $c_4<0$ which
  contradicts the sign pattern of $Q$. So one has
  $e_{d-5}<e_{d-7}$ and conditions (\ref{eq4cond}) imply

  $$e_{d-4}/e_{d-5}~>~\beta >(e_{d-6}-e_{d-4})/(e_{d-7}-e_{d-5})~,~~~\,
  {\rm i.~e.}~~~\, e_{d-4}e_{d-7}~>~e_{d-5}e_{d-6}~.$$
  The left and right symmetric polynomials have

  $$h_1:={d\choose d-4}{d\choose d-7}~~~\, \, 
  {\rm and}~~~\, \, h_2:={d\choose d-5}{d\choose d-6}$$
  terms respectively,
  where

  $$h_2/h_1=h_*:=7(d-4)/5(d-6)>1~,$$
  i.~e. $e_{d-5}e_{d-6}$ has more terms than $e_{d-4}e_{d-7}$. We show that one has
  $e_{d-4}e_{d-7}<e_{d-5}e_{d-6}$ which contradiction implies that the inequality
  $w\leq m-4$ is impossible.

  We set $S_j:=e_j/{d\choose j}$ and then use Newton's inequalities
  $S_j^2\geq S_{j-1}S_{j+1}$ for $j=d-5$ and $j=d-6$. Thus

  $$S_{d-5}^2S_{d-6}^2~\geq ~S_{d-4}S_{d-6}S_{d-5}S_{d-7}~,~~~\, {\rm i.~e.}~~~\,
  S_{d-5}S_{d-6}~\geq ~S_{d-4}S_{d-7}~,~~~\, {\rm so}$$

  $$e_{d-5}e_{d-6}~\geq ~h_* e_{d-4}e_{d-7}~>~e_{d-4}e_{d-7}~.$$
  %$$\begin{array}{ccl}
  %  e_{d-5}e_{d-6}&\geq&\left( {d\choose d-5}{d\choose d-6}/
  %  {d\choose d-4}{d\choose d-7}\right) e_{d-4}e_{d-7}\\ \\
  %  &=&(7(d-4)/5(d-6))e_{d-4}e_{d-7}>e_{d-4}e_{d-7}~.\end{array}$$
  The triples $(u,v,w)$ for which $w\geq m-3$, are all the $15$ triples with
  $0\leq u+v\leq 4$.  
  \vspace{1mm}
  
  Part (2). For $d=6$, the proof can be found in \cite{Kodeg6}. Suppose that
  $d\geq 6$ and that the degree $d$ polynomial $Q$ realizes
  one of the $11$ triples $(u,v,w)$ of
  part (2) of the theorem. Then for $\varepsilon >0$ small enough, the
  polynomial $(1+\varepsilon x)Q$ realizes the triple $(u,v,w+1)$
  for degree $d+1$.
  \vspace{1mm}

  Part (3). 
  Suppose that the polynomial $\tilde{Q}$ realizes one of the triples
  $(4,0,m-3)$, $(3,1,m-3)$, $(2,2,m-3)$ or $(1,3,m-3)$ with the
  sign pattern $\Sigma _{m,2,2}$. Then $\gamma _1<\beta <1=\alpha$. Set
  $A:=(x-\beta )(x+\gamma _1)$ and 

  $$\sum _{k=0}^d\tilde{q}_kx^k=:\tilde{Q}:=AU~,~~~\,
  \sum _{i=0}^{d-2}u_ix^i=:U:=(x-1)
  \prod_{j=2}^{d-2}(x+\gamma _j)~.$$
 The sign pattern of the product $A$ is $(+,-,-)$. The one of $U$ is of the
  form $\Sigma _{\mu ,\nu }$. As
  $\tilde{q}_{\nu +1}=-\beta \gamma _1u_{\nu +1}-(\beta -\gamma _1)u_{\nu}+
  u_{\nu -1}<0$, one must have $\nu +1\leq 3$, i.~e. $\nu =1$ or~$2$.
  For both cases Theorem~1 of \cite{KoPuMaDe} states that the polynomial $U$
  has $\leq 2$ negative roots with 
  moduli smaller than $1$. However one obtains that
  $\gamma _i \leq 1$ for $i=2$, $3$ and~$4$. One perturbs then these roots
  to obtain $\gamma _i<1$ (without changing the signs of the coefficients
  of~$U$) which brings a contradiction with \cite[Theorem~1]{KoPuMaDe}.
\end{proof}

\section{Proof of Theorem~\protect\ref{tmS2n2}\protect\label{secprtmS2n2}}

\subsection{Proof of part (1)}

In the proof of part (1) we use {\em concatenation} of sign patterns and of
couples. The following lemma follows directly from \cite[Lemma~14]{FoKoSh}.

\begin{lm}\label{lmconcat}
Suppose that for the
monic polynomials $P_1$ and $P_2$ of degrees $d_1$ and $d_2$, one has
$\sigma (P_i)=(+,\sigma _i)$, $i=1$,~$2$, where $\sigma _i$
denote what remains of the sign patterns when the initial sign $+$ is deleted.
We set $P^{\dagger}:=\varepsilon ^{d_2}P_1(x)P_2(x/\varepsilon )$.
Then for $\varepsilon >0$
small enough,

$$\sigma (P^{\dagger}):=\left\{
\begin{array}{ll}(+,\sigma _1,\sigma _2)&{\rm if~the~last~position~of~}
  \sigma _1~{\rm is}~+~,\\ \\
  (+,\sigma _1,-\sigma _2)&{\rm if~the~last~position~of~}
  \sigma _1~{\rm is}~-~.\end{array}\right.$$
Here $-\sigma _2$ is obtained from $\sigma _2$
by changing each $+$ by $-$ and vice versa.
\end{lm}

\begin{rem}
  {\rm We use the symbol $\ast$ to denote concatenation of couples or of
    sign patterns. We denote by $\Sigma _{m_1,\ldots ,m_s}$ the sign pattern
    beginning with $m_1$ signs $+$ followed by $m_2$ signs $-$
    followed by $m_3$ signs $+$ etc., so one can write}
  $$\Sigma _{m_1,\ldots ,m_{s-1},m_s+n_1-1,n_2,\ldots ,n_{\ell}}=
  \Sigma _{m_1,\ldots ,m_s}\ast \Sigma _{n_1,\ldots ,n_{\ell}}~.$$
         {\rm When necessary we use more than two consecutive concatenations.
           If $\varepsilon$ is small enough, the moduli of all roots of
           $P_2(x/\varepsilon )$ are smaller than the moduli of all roots of
           $P_1(x)$ which allows to deduce the order of the moduli of
         roots of $P^{\dagger}$.}
\end{rem}

To prove part (1) one takes into account the fact
that the sign pattern $\Sigma _{2,2}$
is realizable with each of the
orders $(0,2)$, $(1,1)$ and $(2,0)$, see \cite[Part~(3) of Example~2]{KoPuMaDe}.
Hence the cases $(\Sigma _{2,n,2},(u,v,w))$ with $u\leq 2$, $w\leq 2$ are
realizable by the triple concatenation

$$(\Sigma _{2,2},(u,2-u))\ast (\Sigma_{n-2},(n-3))\ast (\Sigma _{2,2},(2-w,w))~.$$ 
The concatenation factor in the middle is realizable by any hyperbolic
polynomial with all coefficients positive (and with $n-3$ negative roots).

\subsection{Plan of the proof of part (2)}

We deduce part (2) from two propositions: 

\begin{prop}\label{propfirst}
  Suppose that $n\geq 4$. Then:
  \vspace{1mm}
  
  (1) All couples $(\Sigma _{2,n,2},(u,v,w))$ with either
  $w\geq 5$ or with $w=4$ and $v\geq 1$, are non-realizable.

  (2) All couples $(\Sigma _{2,n,2},(0,n-2,3))$ are non-realizable.
  \end{prop}

The proposition is proved in Subsection~\ref{subsecprpropfirst}.

\begin{prop}\label{prop2n2}
  For $n\geq 4$, the following couples are not realizable:

  (1) $(\Sigma _{2,n,2}, (u,v,4))$, $u+v=n-3$;
  
  (2) $(\Sigma _{2,n,2}, (u,v,3))$, $u+v=n-2$, $u>0$.
  \end{prop}

Part (2) of Proposition~\ref{propfirst} and part (2) of
Proposition~\ref{prop2n2} settle the case $w=3$ while the first parts of these
propositions resolve the case $w\geq 4$. Proposition~\ref{prop2n2}
is proved in Subsection~\ref{subsecprprop2n2}. In its proof three other
propositions are used which are formulated below and proved in
Subsections~\ref{subsecprprop401}, \ref{subsecprprop302}
and \ref{subsecprprop311} respectively.

\begin{prop}\label{prop401}
  The couple $(\Sigma _{2,4,2},(1,0,4))$ is not realizable.
\end{prop}

\begin{prop}\label{prop302}
  The couple $(\Sigma _{2,4,2},(2,0,3))$ is not realizable.
\end{prop}

\begin{prop}\label{prop311}
The couple $(\Sigma_{2,4,2},(1,1,3))$ is not realizable.
  \end{prop}

\subsection{Proof of Proposition~\protect\ref{propfirst}
  \protect\label{subsecprpropfirst}}

  Part (1). Denote by $\beta <\alpha$ the positive and by
  $\gamma _1<\cdots <\gamma _{n+1}$ the
  moduli of the negative roots of a hyperbolic polynomial
  $Q:=\sum _{j=0}^dq_jx^j$ supposed to realize one of the mentioned 
  couples. Denote by $e_j$ the elementary symmetric polynomials
  of the quantities $\gamma _i$. We show that $q_{d-2}>0$ which means that the
  sign pattern of $Q$ is not $\Sigma _{2,n,2}$. Clearly

  $$q_{d-2}=\alpha \beta-(\alpha +\beta )e_1+e_2~.$$

  Suppose that $w\geq 5$. Recall that $u+v+w=d-2=n+1$.
  Set $\ell :=n-w+2=u+v+1$. Then
  $\gamma_{\ell -1}<\alpha <\gamma_{\ell}$.
  For $S:=\sum _{j=\ell}^{n+1}\gamma _j$, one has

  \begin{equation}\label{equabgamma}
    (\alpha +\beta )S<\sum _{\ell \leq i<j\leq n+1}\gamma _i\gamma _j~.
    \end{equation}
  Indeed, the left-hand side contains $2w$ products by two while the right-hand
  side contains $w(w-1)/2\geq 2w$ such products. Besides, for each $k$,
  $\ell \leq k\leq n+1$, it is true that
  $(S-\gamma _k)\alpha <(S-\gamma _k)\gamma _k$ and
  $(S-\gamma _k)\beta <(S-\gamma _k)\gamma _k$.
  Summing up these inequalities yields

  \begin{equation}\label{equRR}
    (w-1)(\alpha +\beta )S<2\sum _{\ell \leq i<j\leq n+1}\gamma _i\gamma _j~.
    \end{equation}

  For each $\nu <\ell$, it is true that
  $(\alpha +\beta )\gamma _{\nu}<(\gamma _{n}+\gamma _{n+1})\gamma _{\nu}$, so

  \begin{equation}\label{equRRR}
    (\alpha +\beta )\sum _{j=1}^{\ell -1}\gamma _j<(\gamma _{n}+\gamma _{n+1})
    \sum _{j=1}^{\ell -1}\gamma _j~.
    \end{equation}
  Equations (\ref{equRR}) and (\ref{equRRR}) imply

    $$(\alpha +\beta )e_1<(2/(w-1))\sum _{\ell \leq i<j\leq n+1}\gamma _i\gamma _j+
    (\gamma _{n}+\gamma _{n+1})\sum _{j=1}^{\ell -1}\gamma _j<e_2$$
    from which $q_{d-2}>0$ follows.
%    Thus the sum of
%  all products by two of moduli of roots in which exactly one of the factors
%  equals $\alpha$ or $\beta$ (the corresponding products of roots are negative)
%  is smaller than the sum of the products by two of the moduli $\gamma _i$.
%  Hence $q_{d-2}>0$.

  Suppose that $w=4$ and $v\geq 1$. For $\mu \leq n-4$, it is true that

  $$(-\alpha -\beta +\gamma _n+\gamma _{n+1})\gamma _{\mu}>0~.$$
  On the other hand, as $\beta <\gamma _{n-3}<\alpha <\gamma _{n-2}$, one has 

  $$\begin{array}{rcl}
    (-\beta +\gamma _{n-3})(\gamma _{n-2}+\gamma _{n-1}+\gamma_n+
    \gamma _{n+1})&>&0~,\\ \\ 
    (-\alpha +\gamma _{n-2})(\gamma _{n-1}+\gamma _n+\gamma_{n+1})&>&0~,\\ \\
    (-\alpha -\beta )\gamma _{n-3}+(\gamma _{n}+\gamma _{n+1})\gamma _{n-1}&>&0~,\\ \\
    -\alpha \gamma _{n-2}+\gamma_n\gamma _{n+1}&>&0
    \end{array}$$
  which again proves that $q_{d-2}>0$, because the left-hand sides of these
  inequalities contain all products by two of moduli in which exactly one of
  the factors equals $\alpha$ or $\beta$ (but not necessarily all products
  $\gamma _i\gamma _j$, and not the product $\alpha \beta$; no product
  $\gamma _i\gamma _j$ is repeated).

  Part (2). Using the same notation one can write:

  $$\begin{array}{lclclcl}
    \alpha \gamma _{n-1}&<&\gamma _{n-1}\gamma _{n}~,&&
    \beta \gamma _{n-1}&<&\gamma _2\gamma _{n}~,\\ \\
    \alpha \gamma _{n}&<&\gamma _{n}\gamma _{n+1}~,&&\beta \gamma _{n}&<&
    \gamma _2\gamma _{n+1}\\ \\
    \alpha \gamma _{n+1}&<&\gamma _{n-1}\gamma _{n+1}~,&&
    \beta \gamma _{n+1}&<&\gamma _1\gamma _{n+1}~,\\ \\
    \alpha \gamma _1&<&\gamma _{n-1}\gamma _1~,&&\beta
    \gamma _{1}&<&\gamma _2\gamma _{1}\\ \\
    \alpha \gamma _2&<&\gamma _{n-1}\gamma_2~,&&\beta
    \gamma _{2}&<&\gamma _n\gamma _{1}
    \end{array}$$
  and for $3\leq k\leq n-2$, $(\alpha +\beta )\gamma _k<
  (\gamma _n+\gamma _{n+1})\gamma _k$. 
All these inequalities together imply $(\alpha +\beta )e_1<e_2$, so $q_{d-2}>0$.

\subsection{Proof of Proposition~\protect\ref{prop401}
  \protect\label{subsecprprop401}}

Suppose that the couple is realizable by a polynomial $Q:=\sum _{j=0}^7q_jx^j$,
  $q_7=1$, with positive roots
  $\alpha$ and $\beta$ and negative roots $-y$ and $-\gamma _i$, where

  \begin{equation}\label{equyba}
    y<\beta <\alpha <\gamma _1<\gamma _2<\gamma _3<\gamma _4~.
    \end{equation}
  As $q_6=-\alpha -\beta +y+\gamma _1+\cdots +\gamma _4$, one has $q_6>0$.
  With roots
  satisfying the above inequalities it is impossible to have $q_1<0$, see
  \cite[Theorem~3]{KoPuMaDe}. It is impossible to have $q_1=0$, $q_2<0$
  either, because
  one can perturb $q_1$ to make it negative by which action the roots remain
  real, distinct and satisfying the above inequalities. So $q_1>0$. 

  We denote by $E_j$ (resp. $G_j$) the $j$th elementary symmetric polynomial
  of the
  quantities $\gamma _i$ (resp. $1/\gamma _i$), $i=1$, $\ldots$,~$4$.
  We show that
  the inequalities (\ref{equyba}) and

  \begin{equation}\label{equq2q5}
    \begin{array}{ccrclc}
      q_2/\alpha \beta yE_4&:=&1/\alpha \beta -(1/\alpha +1/\beta )(G_1+1/y)+
      (1/y)G_1+G_2&<&0&{\rm and}\\ \\
      q_5&:=&\alpha \beta -(\alpha +\beta )(E_1+y)+yE_1+E_2&<&0&\end{array}
    \end{equation}
  cannot simultaneously hold true. In what follows we assume that $\gamma _1=1$
  which can be achieved by a linear change of the variable $x$ and we
  consider instead of the inequalities (\ref{equyba}) the corresponding
  inequalities $\leq$.

  We first observe that for $\alpha +\beta$ fixed, the left-hand sides in the
  inequalities (\ref{equq2q5}) are the smallest possible when the product
  $\alpha \beta$
  is the smallest possible (i.~e. when $\alpha -\beta$ is the maximal possible).
  For the second of these inequalities this is
  evident, for the first of them one has to notice that the coefficient of
  $1/\alpha \beta$ in the left-hand side equals
  $1-(\alpha +\beta )(G_1+1/y)<0$ (because
  $\alpha /y>1$). Thus it suffices to prove the proposition in the two
  extremal cases 
  $\beta =y$ and $\alpha =\gamma _1=1$.

  Suppose that $\beta =y$. Then the second of inequalities (\ref{equq2q5})
  reads:

  $$E_2<\alpha E_1+\beta ^2~.$$
  This is clearly impossible, because $\alpha \leq 1$, $\beta \leq 1$, so
  $\alpha E_1+\beta ^2\leq E_1+1$ while

$$E_2\geq \gamma _2+2\gamma _3+3\gamma _4\geq E_1+2~.$$

  Suppose that $\alpha =\gamma _1=1$. We denote by $e_j$ and $g_j$ the
  elementary
  symmetric polynomials of the quantities $\gamma _i$ and $1/\gamma _i$
  respectively, where $i=2$, $3$, $4$. Thus $E_1=1+e_1$, $E_2=e_1+e_2$,
  $G_1=1+g_1$ and $G_2=g_1+g_2$. The two inequalities (\ref{equq2q5}) read:

  \begin{equation}\label{equq2q5bis}\begin{array}{cclclc}
    (q_2)&:&-1+(\beta -y)g_1-\beta y+\beta yg_2&<&0&{\rm and}\\ \\ 
    (q_5)&:&-1-(\beta -y)e_1-\beta y+e_2&<&0~.&\end{array}
  \end{equation}
  Suppose first that $1-(\beta -y)g_1\geq 0$, i.~e. $\beta -y\leq 1/g_1$.
  Hence the left-hand
  side $L$ in $(q_5)$ satisfies the inequality

  $$L\geq -1-e_1/g_1-\beta y+e_2=(-g_1-e_1-\beta yg_1+e_2g_1)/g_1\geq 0~,$$
  because $g_1\leq 3$, $\beta y\leq 1$, so $-g_1-\beta yg_1\geq -6$, while
  $e_2g_1=2e_1+T$, where

  $$T:=\gamma _2\gamma _3/\gamma _4+
  \gamma _3\gamma _4/\gamma _2+
  \gamma _4\gamma _2/\gamma _3\geq 3(\gamma _2\gamma _3\gamma _4)^{1/3}\geq 3$$
  (by the inequality between the mean arithmetic and the mean geometric),
  so

  $$e_2g_1\geq 2e_1+3~~~\, {\rm and}~~~\,
  L\geq (-e_1-6+2e_1+3)/g_1=(e_1-3)/g_1\geq 0~.$$
Suppose now that $1-(\beta -y)g_1<0$. Then the inequality $(q_2)$ can hold true
only if $\beta y(-1+g_2)<0$, i.~e. if $-1+g_2<0$. This means that for
$\beta -y$ fixed,
the left-hand sides of $(q_2)$ and $(q_5)$ are minimal when
$\beta y$ is maximal.
  This is the case when $\beta =\alpha =1$. Hence the following inequality
  (derived from $(q_5)$) holds true:

  \begin{equation}\label{equey}
    -1-e_1+(e_1-1)y+e_2<0~.
    \end{equation}
  The left-hand side is minimal for $y=0$, because $e_1\geq 3$. Our aim is to
  show that if $\gamma _i\geq 1$, $i=2$, $3$, $4$, and if $g_2<1$, then
  $e_2>e_1+1$; this contradiction with inequality (\ref{equey}) would finish
  the proof of Proposition~\ref{prop401}.

  If $g_2<1$, then as $e_2g_2\geq 9$ (inequality between the mean arithmetic and
  the mean harmonic), one obtains $e_2>9$. So the inequality $e_2<e_1+1$ is
  possible only for $e_1>8$. But for fixed quantity $e_1$, the quantity
  $e_2$ is minimal for $\gamma _2=\gamma _3=1$, $\gamma _4=k>1$. For $e_1>8$,
  one should have $k>6$. 
  For such a choice of
  $\gamma _i$, one gets $e_2-e_1-1=2k+1-k-3=k-2>4>0$.
  This proves the proposition.

  \subsection{Proof of Proposition~\protect\ref{prop302}
    \protect\label{subsecprprop302}}

We suppose that the polynomial $Q:=\sum _{j=0}^7q_jx^j$, $q_7=1$, realizes the
  couple and for the moduli of its roots $\alpha$, $\beta$ and $-y_1$, $-y_2$,
  $-\gamma _1$, $-\gamma _2$, $-\gamma _3$ one has

  \begin{equation}\label{equy1y2ba}
y_1<y_2<\beta <\alpha <\gamma _1<\gamma _2<\gamma _3~.
\end{equation}
  As in the proof of Proposition~\ref{prop401} one shows that $q_6>0$ and
  $q_1>0$. We want to show that one cannot simultaneously have the inequalities
  (\ref{equy1y2ba}), $q_2<0$ and
  $q_5<0$. We denote by $E_j$, $G_j$ the elementary symmetric polynomials of the
  quantities $\gamma _i$ and $1/\gamma _i$ respectively. We set
  $r:=y_1+y_2$ and $t:=y_1y_2$. Then one must have

  \begin{equation}\label{equGrt}
    \begin{array}{ccrccc}
      q_2/\alpha \beta tE_3&:=&1/\alpha \beta-(1/\alpha +1/\beta )(G_1+r/t)+
      G_2+G_1r/t+1/t&<&0&{\rm and}\\ \\
      q_5&:=&\alpha \beta-(\alpha +\beta )(E_1+r)+rE_1+E_2+t&<&0~.&
      \end{array}
    \end{equation}
  For $\alpha +\beta$ fixed, the left-hand sides of the above inequalities
  are minimal
  when $\alpha \beta$ is minimal. In the equation for $q_2/\alpha \beta tE_3$
  the coefficient of $1/\alpha \beta$
  equals $1-(\alpha +\beta)(G_1+r/t)<0$ (because $\alpha r/t>\alpha /y_1>1$).
  So it suffices to
  consider the cases A) $\beta =y_2$ and B) $\alpha =\gamma _1$;
  in both cases one can assume that
  $\gamma _1=1$ which can be achieved by a linear change of the variable~$x$.
  \vspace{1mm}
  
  {\em Case A)}. If $\beta =y_2$, then these inequalities read:

  \begin{equation}\label{equGrtbis}\begin{array}{ccrclc}
      q_2/\alpha \beta tE_3&:=&G_2+G_1/y_1&<&G_1/\alpha +1/\alpha y_1+
      1/\beta ^2&{\rm and}\\ \\ 
q_5&:=&E_2+y_1E_1&<&\alpha E_1+\alpha y_1+\beta ^2~.&\end{array}
    \end{equation}
  One has $\alpha >2/3$. Indeed, if the second of inequalities
  (\ref{equGrtbis}) holds
  true, then it holds true for $y_1=0$, because $E_1>\alpha$.
  But for $y_1=0$ and
  $\alpha \leq 2/3$, it is true that $\alpha E_1+\beta ^2<2E_1/3+1<E_2$.

Next, one has $\gamma _3<2$. Indeed, for $\gamma _3\geq 2$, one gets

$$\begin{array}{rcccc}
  E_2&\geq &2\gamma _1+\gamma _2+\gamma _3&=&E_1+1\\ \\
  &\geq&\alpha E_1-y_1(E_1-\alpha )+1&=&
(\alpha -y_1)E_1+\alpha y_1+1\\ \\ &\geq&
  (\alpha -y_1)E_1+\alpha y_1+\beta ^2&&\end{array}$$
  which is a contradiction with (\ref{equGrtbis}).
  Thus

  \begin{equation}\label{equEG}
    \gamma _1=1\leq \gamma _2<\gamma _3<2~~~\, {\rm and}~~~\,
    4<E_1<5~,~~~\, 2<G_1<3~,~~~\, 5/4<G_2<G_1<3~.
  \end{equation}
  Suppose first that 
  $\alpha -y_1\leq 1/2$. For $E_1$ fixed, consider the function
  $$\Phi (\alpha ,y_1):=(\alpha -y_1)E_1+\alpha y_1$$
  on the closed pentagon

  $$\{ (\alpha ,y_1)~|~0\leq \alpha ,y_1\leq 1~,~\alpha -y_1\leq 1/2~\} ~.$$
  As $\partial \Phi /\partial y_1=-E_1+\alpha <0$, the maximal value of
  $\Phi$ is
  attained on the union of two segments $I_1\cup I_2$, where

$$I_1:=\{ y_1=0,~\alpha \in [0,1/2]\} ~~~\, {\rm and}~~~\,
  I_2:=\{ y_1=\alpha -1/2,~\alpha \in [1/2,1]\} ~.$$
  For $(\alpha ,y_1)\in I_1$, one has
  $\Phi =\alpha E_1\leq E_1/2<E_2-1<E_2-\beta ^2$ which is a
  contradiction with (\ref{equGrtbis}). For $(\alpha ,y_1)\in I_2$, one obtains

  $$\Phi =E_1/2+\alpha (\alpha -1/2)<E_1/2+1/2<E_2-1<E_2-\beta ^2$$
  which is again a contradiction.

  Suppose now that the couple $(\alpha ,y_1)$ belongs to the segment
  $I_3:=\{ \alpha -y_1=\delta ,~\alpha \in [\delta ,1],~\delta \in [1/2,1]\}$.
  Set

  $$\Psi :=(G_1(\alpha -y_1)-1)/\alpha y_1~.$$
  As $G_1>2$ and
  $\alpha -y_1\geq 1/2$, one has
  $\Psi >0$. Moreover, $\Psi >(G_1\delta -1)/(1-\delta )$.
  From (\ref{equGrtbis}) one deduces that $1/\beta ^2>G_2+\Psi$, so

  $$\beta ^2<1/(G_2+\Psi )<1/\left( G_2+\frac{G_1\delta -1}{1-\delta}\right)
  =(1-\delta)/K~,$$
  where

  $$\begin{array}{ccccccc}
    K&:=&G_2(1-\delta )+G_1\delta -1&=&
    G_2+(G_1-G_2)\delta -1&&\\ \\ &>&G_2+(G_1-G_2)/2-1&=&(G_1+G_2-2)/2&>&5/8~,
    \end{array}$$
  see (\ref{equEG}). Thus $\beta ^2<(8/5)(1-\delta )$. One can rewrite the
  second of inequalities (\ref{equGrtbis}) in the form $E_2-\beta ^2<\Phi$.
  However on the segment $I_3$ one has

  $$\Phi =E_1\delta +\alpha (\alpha -\delta )\leq E_1\delta +(1-\delta )<
  E_2-\beta ^2~,$$
  because $E_2=E_2\delta +E_2(1-\delta )$ with $E_1\delta <E_2\delta$ and

  $$(1-\delta )+\beta ^2<(13/5)(1-\delta )<3(1-\delta )<E_2(1-\delta )~.$$
  This contradiction shows that the system of inequalities (\ref{equGrtbis})
  has no solution in Case~A). 
\vspace{1mm}

  {\em Case B)}. Suppose that $\alpha =\gamma _1=1$ and that there exists a
    polynomial $Q:=(x^2-1)Y(x)$ satisfying the conditions of Case B); the
    roots of $Y$ are 
    $\beta$, $-\gamma _2$, $-\gamma _3$, $-y_1$ and $-y_2$. We
    consider a one-parameter family of polynomials $Q_t:=Q-tU$,
    $U:=x^2(x^2-1)(x-\beta )$, $t\geq 0$. The first and last coefficients of $U$
    are positive which means that the coefficients $q_2$ and $q_5$ of $Q_t$
    remain negative for $t\geq 0$; the coefficients $q_0$, $q_1$, $q_6$ and
    $q_7$ do not change.

    As $t$ increases and as long as $Q_t$ remains hyperbolic, 
    \vspace{1mm}
    
    1) the roots $-\gamma_3$ and $-y_2$ move to the left while $-\gamma _2$ and
    $-y_1$ move to the right; $-y_1$ never reaches $0$, because $U(0)=0$, while
    $-\gamma _2$ and $-y_2$ could reach $-1$;
    \vspace{1mm}

     2) the root $-\gamma _3$ cannot go to $-\infty$, because this would mean
    that $q_7=0$;
    \vspace{1mm}
    
    3) neither of
    the coefficients $q_3$ and $q_4$ can vanish. Indeed, for a hyperbolic
    polynomial without root at $0$, it is impossible to have two consecutive
    vanishing
    coefficients, and when a coefficient is $0$, then the two surrounding
    coefficients must have opposite signs (\cite[Lemma~7]{KoMB}).
\vspace{1mm}

It is clear that for $t>0$ sufficiently large, one has $Q_t(2)<0$, so $Q_t$
has more than $2$ positive roots. This can happen only if $Q_t$ is no longer
hyperbolic, because if it is, then it keeps the sign pattern
$\Sigma _{2,4,2}$ (see 3)) and hence has exactly $2$ positive roots.

Loss of hyperbolicity can occur only if the following couple or triple of
roots coalesce: $(-\gamma _2,-1)$, $(-y_2,-1)$ or $(-\gamma _2,-1,-y_2)$.
The triple confluence is a particular case of
$(-\gamma _2,-1)$.
\vspace{1mm}

%The
%case of $(-y_2,-1)$ can be treated as the one of $(-\gamma _2,-1)$ via the
%involution $i_r$ while the triple confluence is a particular case of
%$(-\gamma _2,-1)$.

{\em Case B.1)}. We assume that for $t=t_0>0$, one has
$-\gamma _2=-\gamma _1=-1$.

We consider the one-parameter family $R_s:=Q_{t_0}-sV$, where
$V:=x^2(x+1)^2(x-1)$ and $s\geq 0$.
For $s$ small enough, the sign pattern of $R_s$ is $\Sigma _{2,4,2}$.
As $s$ increases, $-\gamma _3$ moves to the left without reaching $-\infty$,
$-y_2$ and $\beta$ move to the right and $-y_1$ moves to the left.
Hence for some
$s=s_0$, either $\beta$ coalesces with $1$ or $-y_1$ and $-y_2$ coalesce.
\vspace{1mm}

{\em Case B.1.1)}. Suppose that $\beta =1$. Then $R_{s_0}=(x^2-1)^2W$,
where $W:=x^3+Ax^2+Bx+C$ has only negative roots, so $A$, $B$, $C>0$.
Hence

\begin{equation}\label{equsign}
  \begin{array}{lcl}q_2=A-2C<0~,&&q_3=1-2B<0~,\\ \\
    q_4=-2A+C<0&{\rm and}&q_5=B-2<0~.\end{array}
  \end{equation}
The discriminant set 

$$\{ \rho :={\rm Res}(W,W',x)=0\} ~,~~~\,
\rho (A,B,C)=4A^3C-A^2B^2-18ABC+4B^3+27C^2~,$$
separates the set $H_3$ of
hyperbolic polynomials (where $\rho <0$) in
the space $OABC$ from the set of polynomials having exactly one real root
(i.~e. where $\rho >0$). We show that the discriminant set does not intersect
the domain $\mathcal{D}$ in the space $OABC$ defined by conditions
(\ref{equsign}). As
$\rho(1,1,1)=16>0$, the polynomial $R_{s_0}$ is not hyperbolic which is a
contradiction.

\begin{lm}\label{lmnotintersect}
  The set $\{ \rho =0\}$ does not intersect the border of the domain
  $\mathcal{D}:=\{ 0<A/2<C<2A~,~B\in (1/2,2)\}$.
  \end{lm}
The lemma implies that the set $\{ \rho =0\}$ does not intersect the domain
$\mathcal{D}$. Indeed, the set $\{ \rho =0\}$ contains the curve
$A=3t$, $B=3t^2$, $C=t^3$ (polynomials with a triple root at $-t$). This curve
is not contained in $\mathcal{D}$, because it contains points with $C>2A$.
Hence if the set $\{ \rho =0\}$ contains a point from $\mathcal{D}$, then
it contains also a point not from $\overline{\mathcal{D}}$ hence also
a point from the border of $\mathcal{D}$.

\begin{proof}[Proof of Lemma~\ref{lmnotintersect}]
For $A=2C$, one obtains $\rho =32C^4+\tau C^2+4B^3$, $\tau :=-4B^2-36B+27$, whose
discriminant

$$\tau ^2-4\times 32\times 4B^3=(2B-1)(2B-9)^3$$
is negative for $B\in (1/2,2)$. For $A=C/2$, one gets
$\rho =C^4/2+\lambda C^2+4B^3$, $\lambda :=-B^2/4-9B+27$, with discriminant

$$\Delta _0:=\lambda ^2-4\times (1/2)\times 4B^3=(-2+B)(B-18)^3/16$$
which is positive for $B\in (1/2,2)$. However the
biquadratic in $C$ equation $\rho =0$ has no real solution, because
$\Delta _0-\lambda ^2=-8B^3<0$. Thus the discriminant set $\{ \rho =0\}$ does
not intersect the sets $\{ A=2C>0,~B\in (1/2,2)\}$ and
$\{ A=C/2>0,~B\in (1/2,2)\}$.

For $B=1/2$, making use of $C/2<A<2C$, one gets

$$\begin{array}{rclclcl}\rho &=&27C^2+4A^3C-9AC-A^2/4+1/2&&&&\\ \\ &>&
  27C^2+4A^3C-18C^2-C^2+1/2&=&8C^2+4A^3C+1/2&>&0~.\end{array}$$

For $B=2$, one obtains $\rho =27C^2+4A^3C-36AC-4A^2+32$. The derivative 
$\partial \rho /\partial C=4A^3-36A+54C$ takes positive values for
$A\geq 3/2$, $A/2<C<2A$. This follows easily from the fact that the graphs of
the functions $A/2$ and $(-4A^3+36A)/54$ intersect exactly for $A=0$ and
$A=\pm 3/2$. Hence for $A\geq 3/2$ and $A/2\leq C\leq 2A$,
$\rho$ is minimal when $C=A/2$; in this case it equals 

$$2A^4-61A^2/4+32=2(A^2-61/16)^2+375/128>0~.$$
For $0<A\leq 3/2$ and
$A/2\leq C\leq 2A$, the quantity $\rho$ takes its minimal value for $A=3/2$.
Indeed,

$$\partial \rho /\partial A=12A^2C-8A-36C\leq 27C-8A-36C<0~~~\, {\rm for~}~
A\leq 3/2~.$$
But for $A=3/2$, one has $\rho =27C^2-81C/2+23=27(C-3/4)^2+125/16>0$.
Thus the set $\{ \rho =0\}$ does not intersect the border
of the domain~$\mathcal{D}$. 
\end{proof}

\vspace{1mm}

{\em Case B.1.2)}. Suppose that $-y_1$ and $-y_2$ coalesce.
We set $c:=y_1=y_2$ and
we consider the polynomial

$$\sum _{j=0}^7q_jx^j=:Q:=(x+1)^2(x+c)^2(x-1)h(x)~,~~~\, h:=x^2+Ax-B~,$$
where $A=\gamma _3-\beta >0$ and $B=\gamma_3\beta$. As $c\in (0,1)$, one has

\begin{equation}\label{equh}
  h(-1)=1-A-B<0~.
  \end{equation}
The conditions $q_2<0$ and $q_5<0$ read:

\begin{equation}\label{equneqq2q5}
  \begin{array}{lcl}
  q_2=-Ac^2-Bc^2-2Ac+2Bc-c^2+B<0~,&{\rm i.~e.}&B<
  \frac{(c^2+2c)A+c^2}{-c^2+2c+1}~,\\ \\
  q_5=2Ac+c^2+A-B+2c-1<0~,&{\rm i.~e.}&B>(2c+1)A+c^2+2c-1~.
  \end{array}
  \end{equation}
We denote by $(q_2)$, $(q_5)$ and $(g)$ the straight lines in the space $OAB$
defined by the conditions $q_2=0$, $q_5=0$ and $A+B=1$ (see (\ref{equh})).
For $c\in (0,1)$, the slope
$(c^2+2c)/(-c^2+2c+1)$ of the line $(q_2)$ is smaller than the slope
$2c+1$ of $(q_5)$; both slopes are positive. The $A$-coordinates of the
intersection points $(q_2)\cap (q_5)$ and $(q_2)\cap (g)$ equal respectively

$$A':=-(c^4-5c^2+1)/(2c^3-2c^2-2c-1)~~~\, {\rm and}~~~\, A'':=
(1+2c-2c^2)/(4c+1)~.$$
For $c\in (0,1)$, one has

$$A''-A'=9c^2(c^2-2c-1)/((2c^3-2c^2-2c-1)(4c+1))>0~.$$
Hence for $c\in (0,1)$, the half-plane $\{ q_5<0\}$ does not intersect the
sector $\{ q_2<0,~1-A-B<0\}$, so a polynomial $Q$ as above does not exist. 
\vspace{1mm}

{\em Case B.2)}. We assume that for $t=t_0>0$, one has $-\gamma _1=-y_2=-1$.

We consider the one-parameter family $T_s:=Q_{t_0}-sV$, where as above 
$V:=x^2(x-1)(x+1)^2$ and $s\geq 0$. For $s>0$ small enough, the sign pattern
of $T_s$ is $\Sigma _{2,4,2}$. As $s$ increases, $-\gamma _3$ moves to the left
without reaching $-\infty$, $-y_2$ and $\beta$ move to the right while $-y_1$
moves to the left. Hence for some $s=s_1$, one has one of the confluences
$(\beta ,1)$, $(-\gamma _2,-1)$ or $(-y_1,-1)$. In the latter two cases there
is a triple root at~$-1$.
\vspace{1mm}

{\em Case B.2.1)}. Suppose that $\beta =\alpha =1$. Then one can apply the
involution $i_r$ (this does not change the sign pattern $\Sigma_{2,4,2}$) and
obtain a polynomial corresponding to Case~B.1.1) which was already studied.
\vspace{1mm}

{\em Case B.2.2)}. Suppose that $-\gamma _2=-\gamma _1=-y_2=-1$. Then

$$T_{s_1}=(x+1)^3(x-1)(x-\beta )(x+\gamma _3)(x+y_1)=\sum _{j=0}^7q_jx^j~,$$
where $q_2+2q_5=3(y_1-\beta +\gamma _3)>0$. Hence the coefficients $q_2$
and $q_5$ cannot be both negative; the sign pattern of $T_{s_1}$ is not
$\Sigma_{2,4,2}$. 
\vspace{1mm}

{\em Case B.2.3)}. Suppose that $-\gamma _1=-y_2=-y_1=-1$. Then

$$T_{s_1}=(x+1)^3(x-1)(x-\beta )(x+\gamma _2)(x+\gamma_3)=\sum _{j=0}^7q_jx^j~,$$ 
where $q_2+2q_5=3(\gamma _2+\gamma _3-\beta )>0$. Hence the coefficients $q_2$
and $q_5$ cannot be both negative and again the sign pattern of $T_{s_1}$ is not
$\Sigma_{2,4,2}$.

  \subsection{Proof of Proposition~\protect\ref{prop311}
    \protect\label{subsecprprop311}}

   Suppose that the couple is realizable by a hyperbolic polynomial
   $Q:=\sum _{j=0}^7q_jx^j$ the moduli of whose positive roots
   $\beta <\alpha$ and of whose 
  negative roots $-\gamma _i$ satisfy the inequalities

  $$\gamma_1<\beta <\gamma _2<\alpha <\gamma _3<\gamma _4<\gamma_5~.$$
  For the coefficients $q_0=\alpha \beta \gamma _1\cdots \gamma_5$,
  $q_2$ and $q_5$
  of the polynomial $Q$ it is
  true that

  $$\begin{array}{cclcl}
    q_2/q_0&=&\sum _{1\leq i<j\leq 5}1/\gamma_i\gamma_j-
    ((\alpha +\beta )/\alpha \beta )
    \sum _{j=1}^51/\gamma _j+1/\alpha \beta &<&0~,\\ \\
    q_5&=&\alpha \beta-(\alpha +\beta )\sum _{j=1}^5\gamma _j+
    \sum _{1\leq i<j\leq 5}\gamma _i\gamma_j&<&0~.\end{array}$$
  This however is impossible. Indeed, for positive $\alpha$ and $\beta$ and
  when the sum $\alpha +\beta$ is fixed, 
 the product $\alpha \beta$ is minimal when $\alpha$ and $\beta$
  are as far apart as possible. Hence it suffices to consider the two
  extremal situations:
  \vspace{1mm}
  
  1) $\beta =\gamma _1$. In this case

  $$\begin{array}{ccl}
    q_5&=&\sum _{2\leq i<j\leq 5}\gamma_i\gamma_j-\alpha \sum _{j=2}^5\gamma _j
    -\gamma _1^2\\ \\
    &=&(\gamma _3-\alpha )\gamma_2+(\gamma _4-\alpha )\gamma _3
    +(\gamma _5-\alpha )\gamma _4+(\gamma _3-\alpha )\gamma _5\\ \\
    &&+(\gamma _2\gamma _4+\gamma_2\gamma_5-\gamma_1^2)>0
  \end{array}$$
  and the sign pattern of $Q$ is not $\Sigma_{2,4,2}$.
  \vspace{1mm}

   2) $\alpha =\gamma_3$. We assume that $\alpha =\gamma _3=1$.
  For fixed product $\gamma _4\gamma _5$, the sum
  $\gamma _4+\gamma _5$ is the minimal possible when $\gamma _4$ and
  $\gamma _5$ are the closest possible to one another, i.~e. when they are
  equal. One has

  $$\begin{array}{ccll} 
    q_2/q_0&=&\frac{1}{\gamma _4\gamma _5}\left( 1+(\gamma _4+\gamma _5)L_2\right) +
    \psi _2~,&L_2:=
    \frac{1}{\gamma _1}+\frac{1}{\gamma _2}-\frac{1}{\beta} ~,\\ \\
q_5&=&(\gamma _4+\gamma _5)L_5+\gamma _4\gamma _5+
    \psi _5~,&L_5:=\gamma_1+\gamma _2-\beta ~,
  \end{array}$$
  where the quantities $\psi _2$ and $\psi _5$ do not depend on $\gamma _4$
  and $\gamma _5$. As the quantities $L_2$ and $L_5$ are positive, if one has
  $q_2<0$ and $q_5<0$ for some values of the moduli $\gamma _i$ and $\beta$
  such that $\gamma _4\gamma _5=\lambda >0$, then
  these inequalities will hold true also for
  $\gamma _4=\gamma _5=\sqrt{\lambda}$.

   A similar reasoning holds true for the couples $(\gamma_2, \gamma _4)$ and
  $(\gamma _2,\gamma _5)$, but in these cases the modulus $1=\gamma _3=\alpha$
  must remain between the two moduli of the couple, so one can only claim that
  the sum $\gamma_2+\gamma _4$ or $\gamma _2+\gamma _5$ is minimal when one of
  these two moduli equals $1$.

  We apply the reasoning to the couple $(\gamma_2, \gamma _4)$. Hence we can
  assume that either $\gamma _2=1$ or $\gamma _4=1$. If $\gamma _2=1$, then we
  apply the reasoning to the couple $(\gamma _4, \gamma _5)$
  to obtain the case

  $${\rm A)}~~~\, \gamma _2=\gamma _3=\alpha =1~,~~~\, \gamma _4=\gamma _5~.$$
  If $\gamma _4=1$,  then we
  apply the reasoning to the couple $(\gamma _2, \gamma _5)$ to obtain
  the cases

 $$\begin{array}{lll}
    {\rm B)}&\gamma _2=\gamma _3=\gamma _4=\alpha =1&{\rm and}\\ \\
    {\rm C)}&\gamma _3=\gamma _4=\gamma _5=\alpha =1~.&
  \end{array}$$
  In case A) one has to deal with the polynomial
  
  $$(x+\gamma _1)(x-\beta )(x-1)(x+1)^2(x+\gamma _4)^2~~~\, {\rm with}$$

  $$q_2+(1+\gamma _1)q_5=\beta \gamma _4^2(1-\gamma _1)+
  (2\gamma _4-\beta )(\gamma _1^2+\gamma _1)+(2\gamma _4-\beta -1)+
  \gamma _1^2>0~,$$
  so it is impossible to have $q_2>0$ and $q_5>0$.

  In case B) we consider the polynomial
  $$(x+\gamma _1)(x-\beta )(x-1)(x+1)^3(x+\gamma _5)~~~\, {\rm with}$$

  $$q_2+q_5=\gamma _1\beta +\gamma _5(\beta -\gamma _1)+(\gamma _5-\beta )+
  \gamma_1>0~,$$
  so again one cannot have $q_2<0$ and $q_5<0$ at the same time.

  In case C) one considers the polynomial

  $$(x+\gamma _1)(x-\beta )(x+\gamma_2)(x-1)(x+1)^3~~~\,
  {\rm with}$$

  $$q_2+q_5=\beta \gamma_2+\beta \gamma_1+(1-\gamma_2)\gamma_1+
  (\gamma_2-\beta)>0$$
  and one cannot have $q_2>0$ and $q_5>0$.

\subsection{Proof of Proposition~\protect\ref{prop2n2}
                 \protect\label{subsecprprop2n2}}

We prove the proposition by induction on $n$. We prove part (2) first. 
The induction base are the couples
%$(\Sigma _{2,4,2}, (1,0,4))$,
$(\Sigma _{2,4,2}, (2,0,3))$ and
$(\Sigma _{2,4,2}, (1,1,3))$, see
%Propositions~\ref{prop401},
Propositions~\ref{prop302}
and \ref{prop311}. Suppose that it has been proved that the couple
  $(\Sigma _{2,n,2}, (u,v,3))$, $u+v=n-2$, $u>0$, is not realizable.
  This means that for any positive numbers

  \begin{equation}\label{equygamma}
    y_1<\cdots <y_u<\beta <y_{u+1}<\cdots <y_{u+v}<\alpha
  <\gamma _1<\gamma _2<\gamma _3\end{equation}
  (interpreted as the moduli of the roots of a hyperbolic polynomial, the
  positive roots being $\alpha$ and $\beta$) it is impossible
  to simultaneously have

  \begin{equation}\label{equq2q5general}
    \begin{array}{ccrccc}
      q_2&:=&P_*\times (1/\alpha \beta -(1/\alpha +1/\beta )G_1+G_2)&<&0&
      {\rm and}\\ \\
      q_5&:=&\alpha \beta -(\alpha +\beta )E_1+E_2&<&0~,&
      \end{array}
    \end{equation}
  where $P_*:=\alpha \beta \gamma _1\gamma _2\gamma _3y_1\cdots y_{n-2}$,
  $E_j$ (resp. $G_j$) denoting the corresponding elementary symmetric
  polynomials of the quantities $y_i$ and $\gamma _j$ (resp. of $1/y_i$ and
  $1/\gamma _j$). Indeed, the inequalities (\ref{equygamma}) provide the
  positive signs of $q_1$ and $q_6$. If inequalities (\ref{equq2q5general})
  hold true, then as there are just two positive roots, by Descartes' rule
  of signs one must have
  $q_3<0$ and $q_4<0$; the equalities $q_3=0$ and $q_4=0$ are impossible
  by virtue of \cite[Lemma~7]{KoMB}.

  We set
  $\tilde{q}_2:=1/\alpha \beta-(1/\alpha +1/\beta )G_1+G_2$.
  Suppose that one increases $n$ to $n+1$, so one adds a new quantity
  $y_j$ or $\gamma _j$ denoted by
  $\gamma$. Then
  $P_*\mapsto P_*\times \gamma$ and the new quantities $\tilde{q}_2$
  and $q_{d-2}$ equal respectively

  $$\tilde{q}_2+(1/\gamma )(G_1-1/\alpha -1/\beta )~~~\, \, {\rm and}~~~\, \, 
  q_{d-2}+\gamma (E_1-\alpha -\beta )~.$$
  As
  $y_1<\beta$ and $y_2<\alpha <\gamma _1<\gamma _2$, one has
  $E_1-\alpha -\beta >0$
  and $G_1-1/\alpha -1/\beta >0$. This means that both $\tilde{q}_2$ and
  $q_{d-2}$ increase and
  hence after passing from $n$ to $n+1$ they still cannot be both negative.

  To prove part (1) of the proposition one observes first that for $u=0$, it
  follows from part (1) of Proposition~\ref{propfirst}. So one can suppose
  that $u\geq 1$. Then the proof of part (2) is performed in much the same way
  as the proof of part (1). There are only two differences:

  1) the induction base includes also the couple $(\Sigma _{2,4,2}, (1,0,4))$
  which is non-realizable, see Proposition~\ref{prop401};

  2) the inequalities (\ref{equygamma}) have to be replaced by  

  $$ y_1<\cdots <y_u<\beta <y_{u+1}<\cdots <y_{u+v}<\alpha
  <\gamma _1<\gamma _2<\gamma _3<\gamma _4~.$$
  The rest of the reasoning is the same.

\end{document}